\documentclass[12pt]{article}
\usepackage{tikz}
\usepackage{relsize}

\usepackage{empheq}
\usepackage[shortlabels]{enumitem}
\setlist[enumerate]{nosep}
\usepackage[doc]{optional}
\usepackage{xcolor}
\usepackage[colorlinks=true,
            linkcolor=refkey,
            urlcolor=lblue,
            citecolor=red]{hyperref}
\usepackage{float}
\usepackage{soul}
\usepackage{graphicx}
\definecolor{labelkey}{rgb}{0,0.08,0.45}
\definecolor{refkey}{rgb}{0,0.6,0.0}
\definecolor{Brown}{rgb}{0.45,0.0,0.05}
\definecolor{lime}{rgb}{0.00,0.8,0.0}
\definecolor{lblue}{rgb}{0.5,0.5,0.99}

 \usepackage{mathpazo}

\usepackage{stmaryrd}
\usepackage{amssymb}
\newcommand{\seppone}{\setlength{\itemsep}{-1pt}}

\newcommand{\nnn}{\ensuremath{{n\in{\mathbb N}}}}
\newcommand{\thalb}{\ensuremath{\tfrac{1}{2}}}
\newcommand{\menge}[2]{\big\{{#1}~\big |~{#2}\big\}}

\newcommand{\fenv}[1]%
{\ensuremath{\,\overrightarrow{\operatorname{env}}_{#1}}}
\newcommand{\benv}[1]%
{\ensuremath{\,\overleftarrow{\operatorname{env}}_{#1}}}

\newcommand{\scal}[2]{\left\langle{#1},{#2}  \right\rangle}

\newcommand{\RR}{\ensuremath{\mathbb R}}
\newcommand{\RP}{\ensuremath{\mathbb{R}_+}}
\newcommand{\RPP}{\ensuremath{\mathbb{R}_{++}}}
\newcommand{\RM}{\ensuremath{\mathbb{R}_-}}

\newcommand{\NN}{\ensuremath{\mathbb N}}

\newcommand{\inte}{\ensuremath{\operatorname{int}}}

\newcommand{\ran}{\ensuremath{\operatorname{ran}}}

\newcommand{\Id}{\ensuremath{\operatorname{Id}}}

\newcommand{\pinf}{\ensuremath{+\infty}}

{\begin{list}{}{%
\settowidth{\labelwidth}{\textrm{#1~}}%
\setlength{\leftmargin}{\labelwidth+\labelsep}}}%requires macro calc.sty
{\end{list}}
\usepackage{amsthm}
\usepackage[capitalize,nameinlink]{cleveref}
\crefname{lemma}{Lemma}{Lemmas}
\crefname{figure}{Figure}{Figures}
\crefname{equation}{}{equations}
\crefname{chapter}{Appendix}{chapters}
\crefname{item}{}{items}
\crefname{enumi}{}{}
\newtheorem{theorem}{Theorem}[section]
\newtheorem{lemma}[theorem]{Lemma}

\newtheorem{corollary}[theorem]{Corollary}

\newtheorem{example}[theorem]{Example}
\newtheorem{ex}[theorem]{Example}
\newtheorem{fact}[theorem]{Fact}
\newtheorem{remark}[theorem]{Remark}

\providecommand{\RR}{\mathbb{R}}

\providecommand{\ran}{\operatorname{ran}}

\providecommand{\Id}{\operatorname{{ Id}}}

\providecommand{\NN}{\mathbb{N}}

\providecommand{\ran}{\operatorname{ran}}

\providecommand{\Id}{\operatorname{Id}}

\providecommand{\RR}{\mathbb{R}}
\providecommand{\NN}{\mathbb{N}}

\newcommand{\kkk}[1]{\ensuremath{\textstyle\mathsmaller{({#1})}}}

\newcommand{\eins}{\ensuremath{\textstyle\mathsmaller{({1})}}}
\newcommand{\zwei}{\ensuremath{\textstyle\mathsmaller{({2})}}}

\definecolor{myblue}{rgb}{.8, .8, 1}
  \newcommand*\mybluebox[1]{%
    \colorbox{myblue}{\hspace{1em}#1\hspace{1em}}}

\allowdisplaybreaks % or locally if problems {\allowdisplaybreaks
\begin{document}
%-------------------------------------------------------------------------

\title{ \textsc
On Dykstra's algorithm: finite convergence, stalling, 
and the method of alternating projections
}

\author{
Heinz H.\ Bauschke\thanks{
Mathematics, University
of British Columbia,
Kelowna, B.C.\ V1V~1V7, Canada. E-mail:
\texttt{heinz.bauschke@ubc.ca}},~
Regina S.\ Burachik\thanks{
Schoot of IT \& Mathematical Sciences, 
University of South Australia, 
Mawson Lakes, SA, Australia. 
 E-mail: \texttt{regina.burachik@unisa.edu.au}
},~
Daniel B.\ Herman\thanks{E-mail: \texttt{DanielHerman@hotmail.ca}},
~and~
C.\ Yal\c{c}{\i}n Kaya\thanks{
School of Information Technology and Mathematical Sciences, 
University of South Australia, 
Mawson Lakes, Adelaide, SA, Australia.
 E-mail: \texttt{yalcin.kaya@unisa.edu.au}}
}

\date{January 18, 2020}

\maketitle

\begin{abstract}
A popular method for finding the projection onto the intersection 
of two closed convex subsets in Hilbert space is Dykstra's algorithm. 

In this paper, we provide sufficient conditions for Dykstra's algorithm 
to converge rapidly, in finitely many steps. 
We also analyze the behaviour of Dykstra's algorithm 
applied to a line and a square. 
This case study reveals stark similarities to the method of alternating 
projections. Moreover, we show that Dykstra's algorithm 
may stall for an arbitrarily long time. 
Finally, we present some open problems. 
\end{abstract}
{ 

}

\section{Introduction}

Suppose that 
\begin{empheq}[box=\mybluebox]{equation}
\text{$X$ is a Hilbert space, }
\end{empheq}
with inner product $\scal{\cdot}{\cdot}$ and induced norm
$\|\cdot\|$. 
Suppose that 
\begin{empheq}[box=\mybluebox]{equation}
\text{$A$ and $B$ are closed convex subsets of $X$ with $A\cap B\neq\varnothing$, 
and $z\in X$.}
\end{empheq}
Our goal is to find 
\begin{equation}
P_{A\cap B}(z),
\end{equation}
the point in $A\cap B$ nearest to $z$. 
Even when $A$ and $B$ are ``simple'' in the sense that 
$P_A$ and $P_B$ are easily computable, there is in general 
no simple formula for $P_{A\cap B}(z)$. 
Instead, one may employ \emph{Dykstra's algorithm} 
(see \cite{Boyle} and also \cite{Jat},\cite{Deutsch},\cite{BC2017}) 
to find this point. 
The algorithm proceeds as follows.
Set $b_0 := z$,
$p_0 := 0$, 
$q_0 := 0$,
and generate sequences iteratively via
\begin{subequations}
  \label{e:dyk}
\begin{align}
a_n &= P_A(b_{n-1}+p_{n-1}),& p_n &= b_{n-1}+p_{n-1}-a_n,\\
b_n &= P_B(a_n+q_{n-1}),&  q_n &= a_n+q_{n-1}-b_n,
\end{align}
\end{subequations}
where $n\geq 1$. 
The sequences $(a_n)$ and $(b_n)$ are the \emph{main sequences} while
$(p_n)$ and $(q_n)$ are the \emph{auxiliary sequences} of Dykstra's 
algorithm. 
The central convergence result concerning Dykstra's algorithm 
is the following.

\begin{fact}
  \label{f:Dyk}
{\rm\bf (Boyle--Dykstra)}
(See \cite{Boyle}.)
The main sequences $(a_n)$ and $(b_n)$ of Dykstra's algorithm 
both converge strongly to $P_{A\cap B}(z)$. 
\end{fact}

A closely related algorithm is the 
\emph{Method of Alternating Projections (MAP)}, 
which can be thought of as a cousin of Dykstra's algorithm 
with $p_n\equiv q_n \equiv 0$:
We define $c_0 := z$ and proceed via
\begin{equation}
c_{2n-1} = P_A(c_{2n-2})
\;\;\text{and}\;\;
c_{2n} = P_B(c_{2n-1})
\end{equation}
for $n\geq 1$. 

\begin{fact}
{\rm\bf (Bregman)}
(See \cite{Bregman}.)
The MAP sequence 
$(c_n)$ 
converges weakly to some point in $A\cap B$. 
\end{fact}

Note that MAP is simpler than Dykstra's algorithm, but the conclusion 
is also markedly weaker: the convergence is only weak 
(and this indeed can happen, see \cite{Hundal}) and the limit may 
not be $P_{A\cap B}(z)$ (see the next example). 

\begin{example}{\rm\bf (MAP does not produce the projection)}
\label{ex:MAPbad}
Suppose that $X=\RR^2$, 
$A = \RR\times\RM$,
$B = \menge{x=(x\eins,x\zwei)\in\RR^2}{x\eins+x\zwei\leq 0}$,
and $z = (\zeta,\zeta)$, where $\RM = \left]-\infty,0\right]$ and $\zeta > 0$.  
Then $P_{A\cap B}(z) = (0,0)$ while 
$P_A(z) = (\zeta,0)$ and $P_BP_Az = \thalb(\zeta,-\zeta)\in A$. 
Thus, MAP converges in finitely many steps to a point different from
$P_{A\cap B}(z)$ while Dykstra's algorithm follows 
the infinitely many steps of MAP, 
with respect to the boundaries of the sets $A$ and $B$. 
\end{example}

However, when $A$ and $B$ are affine subspaces,
then $(p_n)_\nnn$ lies in $(A-A)^\perp$ and 
$(q_n)$ lies $(B-B)^\perp$; thus, 
the main sequences of Dykstra's algorithm coincide with 
the one produced by MAP in the sense that 
\begin{equation}
(\forall n\geq 1)\quad
c_{2n-1} = a_n 
\;\;\text{and}\;\;
c_{2n} = b_n.
\end{equation}
We record this classical result 
(see Deutsch's monograph \cite{Deutsch} for further information) next.

\begin{fact} {\rm\bf (von Neumann)}
\label{f:vN}
If $A$ and $B$ are closed affine subspaces with nonempty intersection, 
then the MAP sequence coincides with the main sequences of 
Dykstra's algorithm and thus converges strongly to $P_{A\cap B}(z)$.
\end{fact}

\emph{
The goal of this paper is to highlight various behaviours of 
Dykstra's algorithm that have received little attention so far:
(1) we discuss when Dykstra's method converges in finitely many steps;
(2) we exhibit an example where the algorithm stalls 
for an arbitrarily long time; and (3) we provide examples where 
MAP produces the same limit as Dykstra, with less computational 
overload and in fewer steps.
}

The paper is organized as follows.
\cref{s:finite} provides necessary conditions for rapid finite convergence.
In \cref{s:aux}, we develop auxiliary results for the case of a line and
a square. 
Convergence results are presented in \cref{s:main}.
The final \cref{s:conc} contains concluding remarks and 
some open problems.

The notation employed is standard and follows, e.g., \cite{BC2017}
and \cite{Deutsch}. 
A word on notation is in order. 
As in \cref{ex:MAPbad} and also later on, we shall encounter 
vectors and sequences in $\RR^2$. 
If $x$ is such a vector and $(x_n)$ is such a sequence, 
then we write $x=(x\eins,x\zwei)$ and 
$x_n = (x_n\eins,x_n\zwei)$ provided we have a need to refer to their coordinates.

\section{Finite convergence of Dykstra's algorithm}

\label{s:finite}

\begin{lemma}
\label{l:0113a}
Let $n\geq 1$.
Then the following hold: 
\begin{enumerate}
\item 
\label{l:0113a1}
If $b_n=a_n$ 
($\Leftrightarrow q_n=q_{n-1}$), 
then 
$a_{n+1} = b_n$
($\Leftrightarrow p_{n+1}=p_n$).
\item 
\label{l:0113a2}
If $a_{n+1}=b_n$
($\Leftrightarrow p_{n+1}=p_n$),
then 
$b_{n+1} = a_{n+1}$
($\Leftrightarrow q_{n+1}=q_{n}$). 
\end{enumerate}
\end{lemma}
\begin{proof}
All equivalences follow from \eqref{e:dyk}. 
\ref{l:0113a1}:
Suppose $b_n=a_n$.
Then, using also \eqref{e:dyk},
$b_n+p_n = a_n+p_n = b_{n-1}+p_{n-1}$.
Thus 
$a_{n+1} = P_A(b_n+p_n) = P_A(b_{n-1}+p_{n-1}) = a_n=b_n$. 
\ref{l:0113a2}: 
The proof is analogous to that of \cref{l:0113a1}.
\end{proof}

\begin{remark}
If $a_1=b_0$, then it does not necessarily follow that 
$b_1=a_1$.
Indeed, consider any setting where $A$ is not a subset of $B$,
and $z\in A\smallsetminus B$. 
Then $z=b_0=a_1$ and $b_1=P_Ba_1\neq a_1$.
\end{remark}

\begin{corollary}
\label{c:0113a}
Let $n\geq 1$.
\begin{enumerate}
\item 
\label{c:0113a1}
If $b_n=a_n$, 
then $a_n=b_n=a_{n+1} = b_{n+1} = \cdots = P_{A\cap B}(z)$.
\item 
\label{c:0113a2}
If $a_{n+1}=b_n$, 
then $b_n=a_{n+1} = b_{n+1} = a_{n+2} = \cdots = P_{A\cap B}(z)$.
\end{enumerate}
\end{corollary}
\begin{proof}
Combine \cref{f:Dyk} with \cref{l:0113a}.
\end{proof}

The next result provides conditions under which Dykstra's method 
converges almost immediately and where it behaves exactly like MAP. 

\begin{theorem} {\rm\bf (finite convergence)}
\label{t:finite}
Suppose that one of the following holds:
\begin{enumerate}
\item 
\label{t:finite1}
$z-P_A(z)\in N_A(P_BP_Az)$
\item 
\label{t:finite2}
$A$ is affine and $P_BP_Az\in A$.
\end{enumerate}
Then 
$P_{A\cap B}(z)=P_BP_Az = b_1=a_2=b_2=a_3=b_3=\cdots$;
moreover, the main sequences of MAP and Dykstra's algorithm fully coincide. 
\end{theorem}
\begin{proof}
Clearly, $a_1=P_Az$, 
$p_1 = b_0+p_0-a_1 = b_0-a_1 = z-P_Az\in N_A(a_1)$,
$b_1 = P_B(a_1+q_0)=P_Ba_1 = P_BP_Az$, and
$q_1 = a_1+q_0-b_1 = a_1-b_1 \in N_B(b_1)$.
Recall that $a_2=P_A(b_1+p_1)$.

\cref{t:finite1}:
We have 
\begin{subequations}
\begin{align}
z-P_Az\in N_A(P_BP_Az)
&\Leftrightarrow 
p_1 \in N_A(b_1)\\
&\Leftrightarrow 
b_1+p_1\in b_1+N_A(b_1)\\
&\Leftrightarrow 
b_1 = P_A(b_1+p_1)\\
&\Leftrightarrow 
b_1 = a_2.
\end{align}
\end{subequations}
Now apply \cref{c:0113a}\cref{c:0113a2} with $n=1$.

\cref{t:finite2}:
Because $A$ is affine, we have 
$(\forall a\in A)$ $N_A(a) = (A-A)^\perp = \ran(\Id-P_A)$.
Hence 
if $P_BP_Az\in A$, then 
$z-P_Az \in (A-A)^\perp = N_A(P_BP_Az)$
and we are done by \cref{t:finite1}.
\end{proof}

Let us present an example of \cref{t:finite} that was obtained differently 
in \cite{Minh}. 

\begin{example} 
{\rm\bf (cone and ball)}
 (See also \cite[Corollary~7.3]{Minh}.)
\label{ex:Minh}
Suppose $K$ is a nonempty closed convex cone in $X$,
and let $B$ be a multiple of the unit ball.
Then $P_{B\cap K} = P_B\circ P_K$.
\end{example}
\begin{proof}
Let $z\in X$. 
Then there exists $\gamma\geq 0$ such that
$P_BP_Kz = \gamma P_Kz$.
By \cite[Example~6~40]{BC2017},
$N_K(P_BP_Kz) = K^\ominus \cap \{\gamma P_Kz\}^\perp$, 
where $K^\ominus = \menge{x\in X}{\max\scal{x}{K}=0}$ is the dual cone of $K$. 
On the other hand,
$z-P_Kz = P_{K^\ominus}z$ and
$P_{K^\ominus}z \perp P_Kz$; see, e.g., \cite[Theorem~6.30]{BC2017}.
Altogether,
\begin{equation}
 z-P_Kz \in N_K(P_BP_Kz)
\end{equation}
and the result follows from \cref{t:finite}\cref{t:finite1}.
\end{proof}

\begin{remark}
Under the assumptions of \cref{ex:Minh},
it is \emph{not} true that $P_{B\cap K} = P_K \circ P_B$;
see \cite[Example~7.5]{Minh} for more on this.
\end{remark}

\begin{remark}{\rm\bf (two intervals)}
By discussing cases, it is straightforward to verify that 
for any two nonempty closed intervals $A$ and $B$
in $X=\RR$, we have 
\begin{equation}
P_{A\cap B} = P_BP_A = P_AP_B.
\end{equation}
Now consider
\begin{equation}
  A = \left[0,\pinf\right[, 
  \;\;
  B = \left[1,\pinf\right[,\;\;
  \text{and}\;\;
  z \in\RR. 
\end{equation}
If $z\geq 1$, i.e., $z\in A\cap B = B$, then 
$a_n\equiv b_n \equiv z = P_{A\cap B}(z)$
and $p_n\equiv q_n\equiv 0$.
Now assume that $z<1$.
Then $n := -\lfloor z\rfloor \in \NN = \{0,1,2,\ldots\}$ and
$0\leq n+z<1$, where $\lfloor\cdot\rfloor$ denotes the floor function. 
It is tedious but straightforward to verify that 
\begin{equation}
(\forall 1\leq k\leq n)\quad 
a_k = 0,\;\;p_k = z+k-1,\;\; b_k=1,\;\;q_k = -k,
\end{equation}
that 
\begin{equation}
a_{n+1} = z+n,\;\; p_{n+1}=0,\;\; 
b_{n+1} = 1,\;\;q_{n+1} =z-1,
\end{equation}
and that 
\begin{equation}
(\forall k\geq n+2)\quad
a_{k}=b_{k}=1=P_{A\cap B}(z),\;\;
p_{k} = q_{k} = 0.
\end{equation}
In particular, when $z=-1$, we have 
$P_Az = 0$, $z-P_Az=-1$, 
$P_BP_Az = 1\in \inte(A)$, and thus 
$N_A(P_BP_Az) = \{0\}$.
Thus, contrasting to \cref{t:finite}\cref{t:finite1}, 
it is possible to have 
$P_{A\cap B}=P_BP_A$ even though
there exists some point $z\in X$ such that $z-P_A(z)\notin N_A(P_BP_Az)$.
\end{remark}

We conclude this section with a characterization of 
equality of Dykstra and MAP.

\begin{lemma}
\label{l:0113b}
Suppose that $A$ is affine and that $c_{0} = b_0$. 
Then 
\begin{equation}
\label{e:0113b}
(\forall n\geq 2) \quad b_n = P_Ba_n
\end{equation}
if and only if 
Dykstra and MAP coincide, i.e.,
$(\forall n\geq 0)$
$c_{2n} = b_n$ and $c_{2n+1} = a_{n+1}$.
\end{lemma}
\begin{proof}
We always have $c_1=a_1$ and $c_2=b_1$. 
Because $A$ is affine, we also have 
$(\forall n\geq 1)$
$a_n = P_Ab_{n-1}$.

``$\Rightarrow$'':
Because $c_{2} = b_1$, we deduce that 
$c_{3} = P_Ac_{2} = P_Ab_1 = a_{2}$.
In turn, $c_{4} = P_Bc_{3} = P_Ba_{2} = b_{2}$ 
by \eqref{e:0113b}.
Continuing in this fashion, we obtain the conclusion.

``$\Leftarrow$'':
Let $n\geq 2$.
Then $n-1\geq 1$ and 
so 
$b_n = c_{2n} = P_Bc_{2(n-1)+1} = P_Ba_n$. 
\end{proof}

\section{Line and square: set up and auxiliary results}

\label{s:aux}

We assume from now on that 
\begin{empheq}[box=\mybluebox]{equation}
\text{$X=\RR^2$ and $z\in X$,}
\end{empheq}
that 
\begin{equation}
A := \text{is a line in $X$,}
\end{equation}
and that 
\begin{empheq}[box=\mybluebox]{equation}
B := [-1,1]\times[-1,1]
\end{empheq}
is a square of side length $2$ (the unit ball with respect to the max-norm). 
Specifically, in view of symmetry, we also assume that 
\begin{empheq}[box=\mybluebox]{equation}
  u\in A \cap B,
  \;\;
  v \in V := (A-A)^\perp,\;\;\|v\|=1;
  \quad\text{thus,}\quad
  A = u + \{v\}^\perp.
\end{empheq}
and that 
\begin{empheq}[box=\mybluebox]{equation}
v\in\RPP^2,\;\;
-1<u\kkk{1},
\;\;\text{and}\;\;
u\kkk{2} = 1. 
\end{empheq}
(We discuss the case when $v\kkk{1}v\kkk{2}=0$ separately later.)
Then, for every $x\in X$, 
$P_Ax = u+P_{\{v\}^\perp}(x-u) = u+(x-u)-\scal{x-u}{v}v$;
thus, 
\begin{empheq}[box=\mybluebox]{equation}
  \label{e:projs}
P_Ax = x - \scal{x-u}{v}v
\quad\text{and}\quad
P_Bx = P_B(x\kkk{1},x\kkk{2}) = \big(P_{[-1,1]}x\kkk{1},P_{[-1,1]}x\kkk{2}\big).
\end{empheq}
Finally, assume that 
\begin{empheq}[box=\mybluebox]{equation}
\text{$(a_n)$ and $(b_n)$ are the main sequences of Dykstra's algorithm (see \eqref{e:dyk})}
\end{empheq}
while $(p_n)$ and $(q_n)$ are the auxiliary sequences. 
Because $A$ is an affine subspace, the sequence $(p_n)$ lies
entirely in $(A-A)^\perp$ and thus we always have 
\begin{equation}
a_{n} = P_Ab_{n-1},
\end{equation}
where $n\geq 1$; in other words, we can simply ignore $p_{n-1}$ when computing
$a_n = P_A(b_{n-1}+p_n) = P_Ab_{n-1}$. 

In the remainder of this section, we collect
various technical results that will make the proofs
of the main result much simpler. 

\begin{lemma}
  \label{l:1}
Suppose that $b_n = (b_n\kkk{1},1)$, where $b_n\kkk{1}\leq u\kkk{1}$. 
Then
\begin{equation}
  a_{n+1}= b_n+(u\kkk{1}-b_n\kkk{1})v\kkk{1}v.
\end{equation}
\end{lemma}
\begin{proof}
Note that $b_n-u = (b_n\kkk{1}-u\kkk{1},1-1)=(b_n\kkk{1}-u\kkk{1},0)$. 
Hence $\scal{b_n-u}{v} = (b_n\kkk{1}-u\kkk{1})v\kkk{1}\leq 0$.
It follows from \eqref{e:projs} that 
\begin{align}
a_{n+1}
&= 
P_Ab_n = b_n - \scal{b_n-u}{v}v = b_n+(u\kkk{1}-b_n\kkk{1})v\kkk{1}v
\end{align}
as announced. 
\end{proof}

\begin{lemma}
\label{l:2}
Suppose that $b_n = (b_n\kkk{1},1)$, where $n\geq 1$. 
If $b_n\kkk{1}\leq u\kkk{1}$, 
then $a_{n+1}\kkk{2}+q_n\kkk{2} \geq 1$ and thus $b_{n+1}\kkk{2}=1$; 
moreover, 
if the first inequality is strict, then so is the second.
\end{lemma}
\begin{proof}
Recall that $b_n=P_B(a_n+q_{n-1})$.
Thus 
% $a_n(1)+q_{n-1}(1)\leq -1$ 
% and 
$a_n\kkk{2}+q_{n-1}\kkk{2}\geq 1$ and
hence
\begin{equation}
q_{n}\kkk{2} = a_{n}\kkk{2}+q_{n-1}\kkk{2}-b_n\kkk{2} =
a_{n}\kkk{2}+q_{n-1}\kkk{2}- 1 \geq 0.
\end{equation}
On the other hand, \cref{l:1} yields
\begin{equation}
a_{n+1}\kkk{2} = 1 + \big(u\kkk{1}-b_n\kkk{1}\big)v\kkk{1}v\kkk{2}\geq 1.
\end{equation}
Altogether, 
\begin{equation}
  a_{n+1}\kkk{2}+q_n\kkk{2} \geq 1+0=1,
\end{equation}
and the inequality is strict when $b_n\kkk{1}<u\kkk{1}$. 
\end{proof}

\begin{lemma}
\label{l:3}
Suppose that 
$b_1=b_2=\cdots=b_n = (-1,1)$, where $n\geq 1$. 
Then 
\begin{equation}
q_n = (n-1)\big(u\kkk{1}+1\big)v\kkk{1}v + a_1+(1,-1).
\end{equation}
\end{lemma}
\begin{proof}
We verify this using mathematical induction on $n\geq 1$.

Base case: If $(-1,1)=b_1 = P_B(a_1+q_0)=P_B(a_1)$, then 
$q_1 = a_1+q_0-b_1 = a_1-b_1 = a_1+(1,-1)$ as claimed.

Inductive step: Assume that the result holds for some $n\geq 1$ 
and that $(-1,1)=b_1=\cdots=b_n=b_{n+1}$.
By the inductive hypothesis,
\begin{equation}
q_n = (n-1)\big(u\kkk{1}+1\big)v\kkk{1}v + a_1+(1,-1).
\end{equation}
Hence, using also \cref{l:1}, 
\begin{subequations}
\begin{align}
q_{n+1} &= a_{n+1}+q_n-b_{n+1}\\
&=
b_n+(u\kkk{1}-b_n\kkk{1})v\kkk{1}v 
+ (n-1)(u\kkk{1}+1)v\kkk{1}v + a_1+(1,-1) - b_{n+1}\\
&=
(-1,1)+(u\kkk{1}+1)v\kkk{1}v 
+ (n-1)(u\kkk{1}+1)v\kkk{1}v + a_1+(1,-1) + (1,-1)\\
&=
(n+1-1)(u\kkk{1}+1)v\kkk{1}v + a_1+(1,-1),
\end{align}
\end{subequations}
as required. 
\end{proof}

\begin{lemma}
  \label{l:4}
Suppose that $b_n\kkk{1}\leq u\kkk{1}$ and $b_n=(b_n\kkk{1},1)$, 
where $n\geq 1$. 
Then
\begin{equation}
b_{n+1} = (-1,1)
\quad\Leftrightarrow\quad
a_{n+1}\kkk{1}+q_n\kkk{1} \leq -1.
\end{equation}
\end{lemma}
\begin{proof}
Recall that $b_{n+1} = P_B(a_{n+1}+q_n)$.

``$\Rightarrow$'':
If $b_{n+1} = (-1,1)$, 
then, 
since $b_{n+1}=P_B(a_{n+1}+q_n)$ 
and $b_{n+1}\kkk{1} =-1$, we have 
$a_{n+1}\kkk{1}+q_n\kkk{1}\leq -1$.

``$\Leftarrow$'':
Clear from \cref{l:2}.
\end{proof}

\begin{lemma}
\label{l:5}
Suppose that $b_1=\cdots = b_n = (-1,1)$, where $n\geq 1$. 
Then $q_n\kkk{1}\leq 0$, 
\begin{equation}
  \label{e:1911092am}
a_{n+1}\kkk{1}+q_n\kkk{1} = a_1\kkk{1}+n\big(u\kkk{1}+1\big)v^2\kkk{1}
\end{equation}
and 
\begin{equation}
  \label{e:191109c}
  a_{n+1}\kkk{1}+q_n\kkk{1}\leq a_{n+1}\kkk{1}<u\kkk{1};
\end{equation}
moreover, 
\begin{equation}
  \label{e:1911092am+}
b_{n+1} = (-1,1)
\quad
\Leftrightarrow
\quad
n\big(u\kkk{1}+1\big)v^2\kkk{1}+a_1\kkk{1}\leq -1
\quad
\Leftrightarrow
\quad
n\leq \left\lfloor \frac{\displaystyle -1-a_1\kkk{1}}{\displaystyle(u\kkk{1}+1)v^2\kkk{1}} \right\rfloor. 
\end{equation}
If $b_{n+1}\neq (-1,1)$, 
then 
$-1 < a_{n+1}\kkk{1} + q_n\kkk{1} \leq a_{n+1}\kkk{1}<u\kkk{1}\leq 1$,
$b_{n+1}\kkk{1} = a_{n+1}\kkk{1} + q_n\kkk{1}$, 
$b_{n+1}\kkk{2} = 1$, 
and
$q_{n+1}\kkk{1} = 0$. 
\end{lemma}
\begin{proof}
Clearly, $q_n \in N_B(b_n) = \RM\times \RP$, so $q_n\kkk{1}\leq 0$. 
Because $b_1 = P_B(a_1+q_0)=P_Ba_1$, it is clear that $a_1\kkk{1}\leq -1$ and so 
\begin{equation}
  -1-a_1\kkk{1}\geq 0.
\end{equation}
From \cref{l:3}, we have 
\begin{equation}
q_n = (n-1)(u\kkk{1}+1)v\kkk{1}v + a_1+(1,-1);
\end{equation}
in particular, 
\begin{equation}
q_n\kkk{1} = (n-1)(u\kkk{1}+1)v^2\kkk{1} + a_1\kkk{1}+1. 
\end{equation}
From \cref{l:1}, we have 
\begin{equation}
  a_{n+1}\kkk{1}= -1+(u\kkk{1}+1)v^2\kkk{1}.
\end{equation}
Adding the last two equations gives \eqref{e:1911092am}.
Note that $-1 < u\kkk{1}$
$\Leftrightarrow$
$v^2\kkk{1}-1 < u\kkk{1}(1-v^2\kkk{1})$
$\Leftrightarrow$
$-1+u\kkk{1}v^2\kkk{1}+v^2\kkk{1} < u\kkk{1}$
$\Leftrightarrow$
$a_{n+1}\kkk{1}<u\kkk{1}$, which gives
\eqref{e:191109c} because $q_n\kkk{1}\leq 0$. 
  
On the other hand, from \cref{l:4}, we have
\begin{equation}
b_{n+1} = (-1,1)
\quad\Leftrightarrow\quad
a_{n+1}\kkk{1}+ q_n\kkk{1} \leq -1.
\end{equation}

Therefore, using \eqref{e:1911092am}, we obtain 
\begin{subequations}
\begin{align}
b_{n+1} = (-1,1)
&\quad\Leftrightarrow\quad
n(u\kkk{1}+1)v^2\kkk{1} + a_1\kkk{1} \leq -1\\
&\quad\Leftrightarrow\quad
n \leq \frac{-1-a_1\kkk{1}}{(u\kkk{1}+1)v^2\kkk{1}},
\end{align}
\end{subequations}
and \eqref{e:1911092am+} follows.

Now assume that $b_{n+1} \neq (-1,1)$.
By \cref{l:4}, 
\begin{equation}
 -1< a_{n+1}\kkk{1}+q_n\kkk{1}.
\end{equation}
But we know already that 
$a_{n+1}\kkk{1}+q_n\kkk{1} \leq a_{n+1}\kkk{1}<u\kkk{1}\leq 1$.

The formula for $b_{n+1}\kkk{1}$ is now clear. 
The statement that $b_{n+1}\kkk{2}=1$ is a consequence of \cref{l:2}. 
\end{proof}

\begin{lemma}
\label{l:5+}
Suppose $n\geq 1$,
$-1<b_n\eins<u\eins$, 
$b_n\zwei=1$, 
$a_n\eins<u\eins$,
and 
$q_{n-1}\eins\leq 0$.
Then
$-1<a_{n}\eins+q_{n-1}\eins < b_{n+1}\eins = a_{n+1}\eins<u\eins$,
$b_{n+1}\zwei = 1$, 
$q_n\eins=0$, 
$q_n\zwei\geq 0$,
and 
$b_{n+1}=P_Ba_{n+1}$. 
\end{lemma}
\begin{proof}
We have $q_n\in N_B(b_n)$ and so 
$q_n\eins = 0$ and $q_n\zwei\geq 0$. 
Hence 
$b_{n+1}\eins = (P_B(a_{n+1}+q_{n}))\eins = (P_Ba_{n+1})\eins$. 
We also have
$b_n\eins = a_n\eins+q_{n-1}\eins$ 
because $-1<b_n\eins<1$. 
Now $a_{n+1} = 
b_n + (u\eins-b_n\eins)v\eins v$ 
by \cref{l:1}.
On the one hand,
\begin{subequations}
\begin{align}
a_{n+1}\eins 
&= a_n\eins+q_{n-1}\eins + 
(u\eins-(a_n\eins+q_{n-1}\eins))v^2\eins \\
&=
(1-v^2\eins)\big(a_n\eins+q_{n-1}\eins\big)
+v^2\eins u\eins
\end{align}
\end{subequations}
thus 
\begin{equation}
-1<a_{n}\eins+q_{n-1}\eins < a_{n+1}\eins
=a_{n+1}\eins+q_n\eins < u\eins \leq 1
\end{equation}
and 
$b_{n+1}\eins = a_{n+1}\eins \in \left]-1,1\right[$. 
On the other hand,
\begin{subequations}
\begin{align}
a_{n+1}\zwei
&= b_n\zwei+(u\eins-b_n\eins)v\eins v\zwei\\
&= 1+(u\eins-b_n\eins)v\eins v\zwei\\
&> 1
\end{align}
\end{subequations}
and thus $a_{n+1}\zwei+q_n\zwei \geq a_{n+1}\zwei>1$
which yields $b_{n+1}\zwei  = (P_Ba_{n+1})\zwei = 1$ 
and $q_{n+1}\zwei\geq 0$.
Altogether, $b_{n+1} = P_Ba_{n+1}$.
\end{proof}

\begin{corollary}
\label{c:7}
Suppose that 
$-1<b_n\eins<u\eins$, 
$b_n\zwei=1$, 
$a_n\eins<u\eins$, 
and 
$q_{n-1}\eins\leq 0$, where $n\geq 1$. 
Then for every $k\geq 1$, we have
$b_{n+k} = P_B(a_{n+k})$,
$-1<b_{n+k-1}\eins < b_{n+k}\eins < u\eins$ and 
$b_{n+k-1}\zwei=1$.
In other words, starting with $a_{n+1}$,
the main sequences of Dykstra coincide with 
the MAP sequence 
(starting at $a_{n+1}$) and all converge to 
$P_{A\cap B}z$, which is $u$ in this setting. 
\end{corollary}
\begin{proof}
This follows inductively from \cref{l:5+}. 
Notice that $\lim_{n\to\infty} b_n=u$,
because $(b_m)_m$ converges to $P_{A\cap B}z$ and 
$(b_m)_m$ lies eventually in $B\cap [-1,1]\times\{1\}$. 
So the limit lies in $A\cap B \cap [-1,1] \times\{1\} = \{u\}$.
\end{proof}

\section{Line and square: main results}

\label{s:main}

We are now ready to describe our main results for the line-square setting. 
There are essentially three scenarios, \emph{depending on the starting point $z$}, 
for Dykstra's algorithm:
(1) rapid finite convergence; 
(2) infinite convergence with steady progress; 
(3) initial stalling followed by infinite convergence with steady progress.
These three regions are depicted in \cref{fig:regions}, 
and we discuss them in the subsections below. 
\begin{figure}[h!]
  \begin{center}
  \includegraphics[scale=1.8]{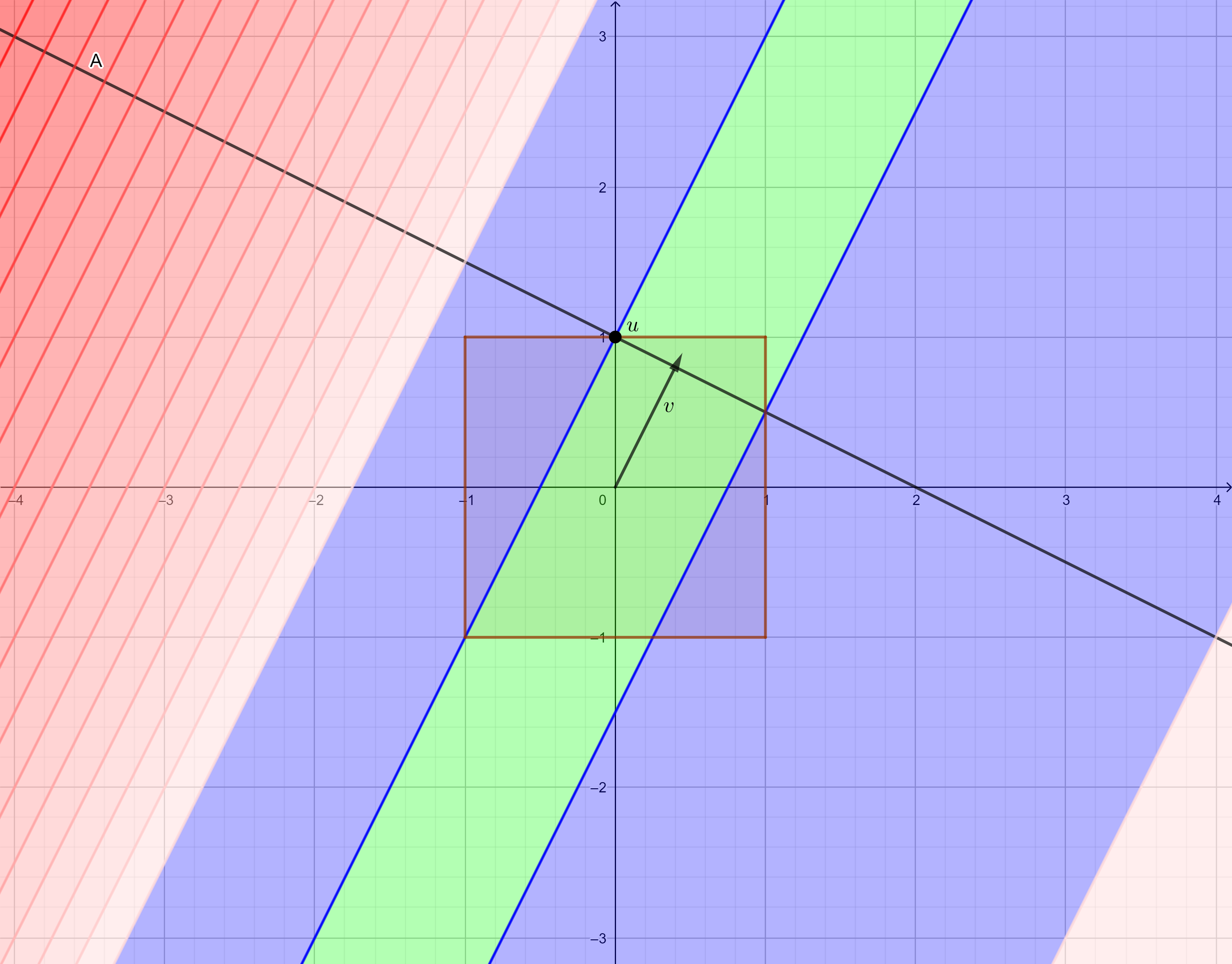}
  \caption{Three scenarios are possible, 
  depending on the location of the starting point $z$.
  If $z$ belongs to the green region 
  containing the origin, then Dykstra's algorithm converges rapidly 
  in finitely many steps. If $z$ belongs to one of the two 
  adjacent blue regions still intersecting the square $B$, 
  then Dykstra's algorithm does not
  converge finitely and it coincides with MAP. Finally, if $z$ is in the 
  remaining red region, the stalling occurs followed by infinite progress. 
  See the subsections in \cref{s:main} for details.}
  \label{fig:regions}
  \end{center}
\end{figure}
As we shall see, there is a close relationship to MAP.

\subsection*{When Dykstra's algorithm and MAP coincide, 
with rapid finite convergence}
We are done in two steps  provided that $P_A(z)\in B$: 

\begin{theorem} {\rm\bf (${P_Az\in B}$)}
Suppose that $u\eins\leq a_1\eins\leq 1$ and $|a_1\zwei|\leq 1$. 
Then the main sequences of Dykstra's algorithm coincides with 
the MAP sequence and convergence is finite and rapid:
$P_{A\cap B}(z) = a_1=b_1=a_2=b_2=\cdots$.
\end{theorem}
\begin{proof}
The hypothesis implies that $a_1=P_Az\in B$.
Hence $b_1 = P_BP_Az=P_Ba_1 = a_1 = P_Az\in A$ and the result
follows from \cref{t:finite}\cref{t:finite2}.
\end{proof}

We now turn to the case we omitted in the previous section 
--- the case when the line is parallel to a side of the box. 
It turns out that this also leads to 
finite convergence although two steps may be required. 

\begin{theorem} {\rm\bf (parallel case)}
Suppose that $A = \RR\times\{\alpha\}$, where $|\alpha|\leq 1$.
Then the main sequences of Dykstra's algorithm  and the MAP sequence
coincide; moreover, 
$P_{A\cap B}(z) = b_1=a_2=b_2=a_3=\cdots$.
\end{theorem}
\begin{proof}
The hypothesis implies that $P_BP_Az\in A$.
Now apply \cref{t:finite}\cref{t:finite2}.
\end{proof}

\subsection*{When Dykstra's algorithm and MAP coincide with infinite convergence}

\begin{theorem} {\rm\bf (Dykstra $\equiv$ MAP)}
  \label{t:hard}
Suppose that $-1<a_1\eins<u\eins$ and $1 < a_1\zwei$.
Then Dykstra's algorithm and MAP produce the exactly same main sequences,
with infinite convergence. 
\end{theorem}
\begin{proof}
The hypothesis implies that $-1<b_1\eins = a_1\eins<u\eins$ and
$b_1\zwei=1$. Recall also that $q_0=0$. 
The conclusion thus follows from \cref{c:7}, with $n=1$.
\end{proof}

\begin{remark}
We saw in \cref{t:hard} directly that MAP and Dykstra's algorithm do 
not converge in finitely many steps.
In fact, this is a universal phenomenon of MAP because
Luke, Teboulle, and Thao recently proved (see \cite[Theorem~7]{Luke}) that
\emph{in general} we have the dichotomy that
either $P_BP_A(z)\in A$ (and MAP terminates) or 
MAP does not converge in finitely many steps. 
\end{remark}

\subsection*{When Dykstra's algorithm stalls}

\begin{theorem}{\rm\bf (stalling)}
\label{t:stall}
Suppose that 
$a_1\eins\leq-1$ and $1< a_1\zwei$. 
Set 
\begin{equation}
n := 1+ \left\lfloor \frac{\displaystyle -1-a_1\kkk{1}}{\displaystyle(u\kkk{1}+1)v^2\kkk{1}} \right\rfloor.
\end{equation}
Then Dykstra algorithm stalls, i.e.,  $b_1 = b_2 = \cdots = b_n = (-1,1)$,
it then ``breaks free'' with $b_{n+1}\neq (-1,1)$, and 
it finally acts like MAP with starting point $b_{n+1}$.
\end{theorem}
\begin{proof}
Combine \cref{l:5} with \cref{c:7}. 
\end{proof}

\begin{remark}
  \label{r:last}
  Some comments regarding \cref{t:stall} are in order.
\begin{enumerate}
  \item By choosing $a_1=z\in A$ with $z\eins$ very negative,
  we can arrange for $n$ to be as large as we want. 
  Thus the stalling phase for Dykstra's algorithm can be arbitrarily long!
  \item 
  \label{r:last2} 
  The point $b_{n+1}$ is not necessarily equal to $P_Ba_{n+1}=(a_{n+1}\eins,1)$;
  in fact, with the help of  \cref{l:1} and \cref{l:5}, one obtains
  \begin{equation}
   -1 <b_{n+1}\eins = a_1(\eins)+n(u\eins+1)v^2\eins\leq a_{n+1}\eins.
  \end{equation}
  Somewhat surprisingly, the orbits (in the sense of sets) of Dykstra's algorithm 
  and MAP need not be identical --- see \cref{fig:orbits} 
  for a visualization.
\end{enumerate}
\end{remark}

\begin{figure}[h!]
  \begin{center}
  \includegraphics[scale=3.0]{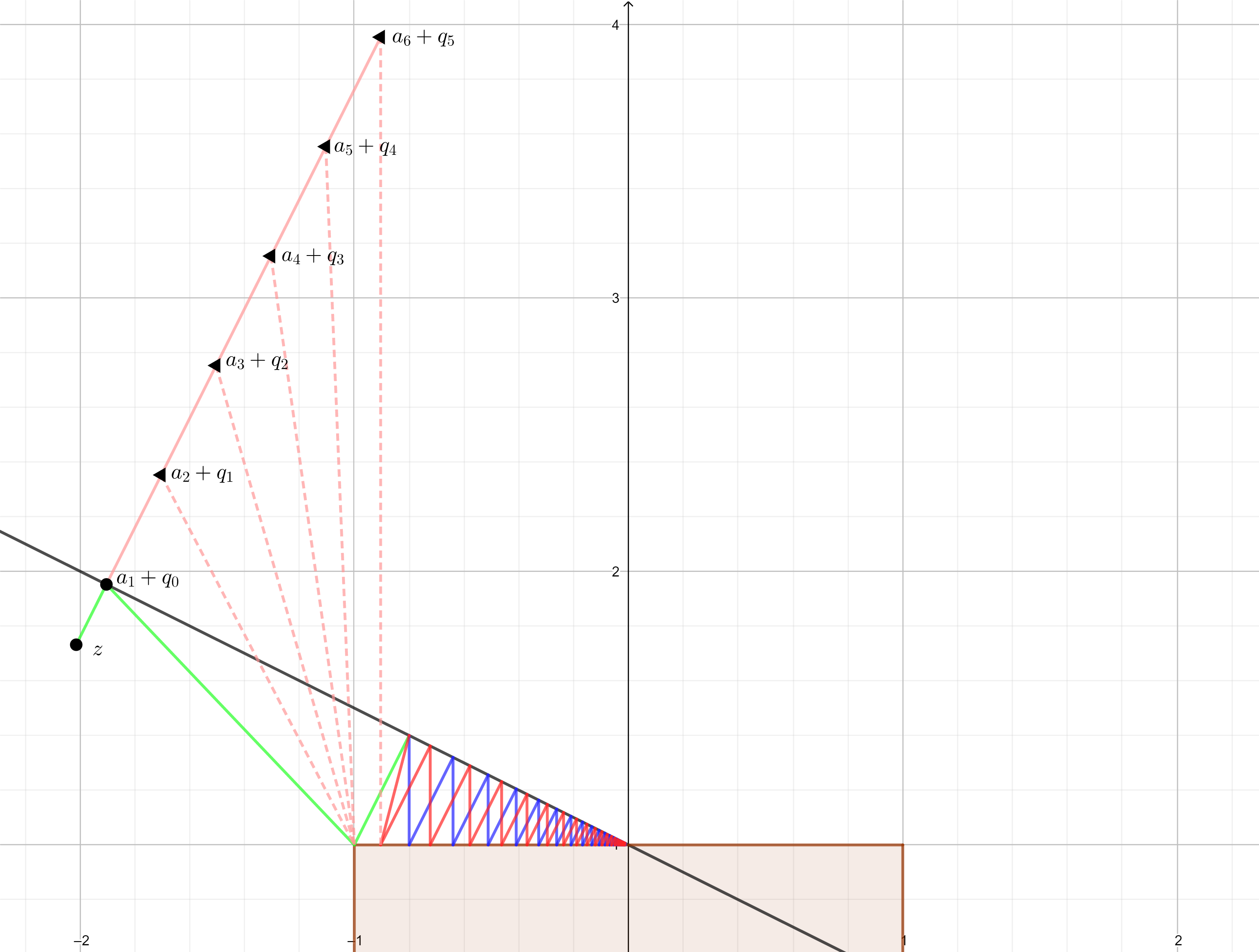}
  \caption{An illustration of \cref{r:last}\ref{r:last2} 
  where the starting point $z$ lies in the stalling region.
  Here $b_1 = b_2 = \cdots = b_5 = (-1,1)$ illustrates stalling;
  the orbit until this point is depicted in green. 
  (The stalling period can be made arbitrarily long by, for instance, 
   moving the starting point $z$ to the left.)
  Dykstra's algorithm then exits the stalling period; however, 
  $b_6$ is \emph{not} equal to $P_B(a_6)$! From this point 
  onwards, Dykstra's algorithm proceeds like MAP but starting from $b_6$, 
  with its orbit depicted in red. 
  In contrast, MAP proceeds along the green and then blue orbit, 
  without any stalling. 
  }
  \label{fig:orbits}
  \end{center}
\end{figure}

\section{Conclusion}

\label{s:conc}

The following example underlines 
the importance of the \emph{order} of the sets --- 
projecting first onto the square and then onto the line will not work!

\begin{example} {\rm\bf (order matters!)}
Suppose 
that $A$ is the line through the points $(0,1)$ and $(1,0)$, and that 
$B = [-1,1]\times[-1,1]$ is the square in $\RR^2$. 
Consider $z=(-2,-1)$.
Then $P_Bz = (-1,-1)$ and thus $P_AP_Bz = (\thalb,\thalb)\in B$
while $P_{A\cap B}z = (0,1)$.
Hence MAP stops right away with the limit being
different from $P_{A\cap B}(z)$, the limit 
of the main sequences of Dykstra's algorithm. 
\end{example}

The following questions appear to be of interest and 
are left for future investigations. 
\begin{itemize}
  \item Can we identify more cases when it suffices to apply
   MAP to find $P_{A\cap B}(z)$?
  \item 
  \marginpar{Reworded} 
  Can one prove a higher-dimensional version of 
  the box-line scenario considered in the second half of this paper? 
  In fact, \cite{BBK} suggests that 
  \eqref{e:0113b} holds numerically and thus 
  that extensions may be possible. 
  \item If MAP and Dykstra's algorithm yield the same limit, is it true
  that the convergence of MAP is never slower than Dykstra?
  All results in this paper --- as well as those in \cite{Jat} --- suggest 
  that this is true for some classes of problems.
\end{itemize}

\section*{Acknowledgments}
HHB was partially supported by the Natural Sciences and
Engineering Research Council of Canada.

\end{document}